\documentclass[10pt,a4paper]{article}
\usepackage[english]{babel}
\usepackage{amsmath}
\usepackage{amsxtra, amsfonts, amssymb, amstext}
\usepackage{amsthm}
\usepackage{booktabs}
\usepackage{fullpage}
\usepackage{nicefrac}
\usepackage{xspace}
\usepackage[noadjust]{cite}
\usepackage{url}
\usepackage{graphics}
\usepackage{color}
\usepackage[algo2e,ruled,vlined,linesnumbered]{algorithm2e}
\usepackage{multirow}
\usepackage{float}
\usepackage{enumerate}
\usepackage[symbol]{footmisc}

\newtheorem{theorem}{Theorem}
\newtheorem{lemma}[theorem]{Lemma}

\DeclareMathOperator{\Bin}{Bin}

\def\N{{\mathbb N}}
\def\R{{\mathbb R}}
\newcommand{\Exp}{\mathbb{E}}
\def\eps{\varepsilon}
\def\S{{\cal S}}

\begin{document}

\title{Drift Analysis and Evolutionary Algorithms Revisited}

\author{Johannes Lengler,
Angelika Steger}
\date{
\ \\*[-2em]
ETH Z\"urich, Department of Computer Science, Switzerland\\
\bigskip
{\em In memory of Ji\v{r}\'i Matou\v{s}ek}}

\maketitle

\begin{abstract}
One of the easiest randomized greedy optimization algorithms is the following evolutionary algorithm which aims at maximizing a boolean {function} $f:\{0,1\}^n \to {\mathbb R}$. The algorithm starts with a random search point $\xi \in \{0,1\}^n$, and in each round it flips each bit of $\xi$ with probability $c/n$ independently at random, where $c>0$ is a fixed constant. The thus created {offspring} $\xi'$ replaces $\xi$ if and only if $f(\xi') \geq f(\xi)$. The analysis of the runtime of this simple algorithm for monotone and for linear functions turned out to be highly non-trivial. In this paper we review  known results and provide new and self-contained proofs of partly stronger results. 
\end{abstract}

\section{Introduction}
Evolutionary algorithms have a long and successful tradition in solving real world optimization problems. Despite much effort the theoretical understanding of these algorithms, however, is still very limited, even for seemingly `trivial' versions.  
One of the easiest evolutionary algorithms aims at maximizing a pseudo-boolean {function} $f:\{0,1\}^n \to {\mathbb R}$. This algorithm starts with a random search point $\xi \in \{0,1\}^n$, and in each round it flips each bit of $\xi$ with probability (or \emph{mutation rate}) $c/n$ independently, where $c>0$ is the so-called \emph{mutation parameter}. The thus created \emph{offspring} $\xi'$ replaces $\xi$ if and only if $f(\xi') \geq f(\xi)$. Naturally, the interesting question is to determine the \emph{optimization time} of $f$, i.e.,  the number of rounds until a maximum of $f$ is found (in expectation). Even for the seemingly trivial case that $f$ is a strictly monotone function (that is, $f(\xi) > f(\xi')$ for all $\xi,\xi'$ so that $\xi\not=\xi'$ and $\xi_i\ge \xi_i'$ for all $1\le i\le n$) determining the asymptotic optimization time turned out to be far from trivial. One of the first rigorous results in this direction was~\cite{Muehlenbein92} who determined the optimization time for the case that $f(\xi)=\sum_{i=1}^n\xi_i$. For general linear functions it required substantial efforts \cite{Droste2002,HeYao:04:drift,jagerskupper2011combining,doerr2012non,Doe-Joh-Win:j:12:multiDrift} until  Doerr and Goldberg~\cite{DoerrGoldberg:j:11:adaptiveDrift} finally showed that the optimization time is $\Theta(n\log n)$ for all constant mutation parameters $c>0$. 
For general monotone functions it is easy to see, cf. e.g.~\cite{doerr2013mutation}, that the optimization time is  $\Theta(n\log n)$ for all constants $0<c<1$. However, as it turned out, this is not necessarily so for larger mutation parameters.
Doerr et al.~\cite{doerr2013mutation} showed that there are monotone functions such that for $c>16$ the algorithm takes exponential time. 

In  this paper we provide short and elegant proofs for various drift theorems that will allow us to
give simple proofs for the above statements which are partly stronger than previous results. The exact results can be found in Sections~\ref{sec:monotone} and~\ref{sec:linear}. The paper is completely self-contained. 

\section{Basic Tools}\label{drift:theorems}

The main tool in the proofs (both in previous proofs and in ours) is drift analysis. Going back to the seminal paper of Hajek~\cite{Haj:j:82}, drift analysis has been developed as a convenient tool to estimate the expected hitting time of algorithms and random processes (see~\cite{kotzing2016concentration} for an overview). 
In this section we collect various known statements and provide short and self-contained proofs. Our emphasis is on elegance of proofs rather than obtaining general statements. We always assume the following setup. We study a Markov chain $(X_t)_{t \in\N_0}$ over some state space $\S$, which we sometimes restrict to subsets of $\{0\} \cup [1,\infty)$, or simply to $\S = \N_0$. With $T$ we denote the random variable that denotes the earliest point in time $t$ such that $X_t = 0$. We are usually interested in determining the expectation $\Exp[T]$ as a function of the initial state $X_0=n$. As we are only interested in the first time where $X_t$ reaches zero we may (and do) assume without loss of generality  that $X_{T+1} = X_{T+2} = \ldots = 0$. 

Our first theorem is a reformulation of Wald's equation, cf.\ also He and Yao~\cite{HeY01}.
\begin{theorem}\label{thm:drift:t1}
Let $(X_t)_{t \in\N_0}$ be a Markov chain with state space ${\cal S} \subseteq [0,\infty)$ and assume $X_0=n$. Let $T$ be the earliest point in time $t \geq 0$ such
that $X_t  = 0$. If there exists $c > 0$ such that for all $x\in {\cal S}, x > 0$ and for all $t\geq 0$ we have
\begin{equation}\label{eq:additivedrift}
\Exp[X_{t+1}\mid X_t=x]\le x - c,
\end{equation}
then
$$
\Exp[T] \leq \frac{n}{c}.
$$
\end{theorem}
\begin{proof}
One easily checks that the condition of the theorem implies that for all $t \in\N_0$
\begin{equation}\label{eq:additivedrift2}
\Exp[X_{t+1}\mid T>t]\le \Exp[X_{t}\mid T>t]- c.
\end{equation}
By conditioning on whether $T>t$ or $T\le t$ we  obtain 
\begin{align}\label{eq:drift:eq}
\Exp[X_t] &= \Pr[T>t] \cdot \Exp[X_t \mid T>t] + \Pr[T\le t] \cdot 0
\;=\; \Pr[T>t] \cdot \Exp[X_t \mid T>t].
\end{align}
Proceeding similarly for $X_{t+1}$  we obtain 
\begin{align}\label{eq:drift:eq2}
\Exp[X_{t+1}]&\;\;=\;\;\Pr[T>t] \cdot\Exp[X_{t+1} \mid\! T>t] + \Pr[T\le t]\cdot 0\nonumber\\
& \stackrel{(\ref{eq:additivedrift2})}\le    \Pr[T>t] \cdot (\Exp[X_{t} \mid T>t] -  c )\nonumber\\ 
& \stackrel{(\ref{eq:drift:eq})}=   \Exp[X_t] - \Pr[T >t] \cdot c.
\end{align}
Since $T$ is a random variable that takes values in $\N_0$, we may write $\Exp[T] = \sum_{t=0}^{\infty}\Pr[T>t]$. Thus 
\begin{align}\label{eq:drift:eq3}
c\cdot\Exp[T] 
& \stackrel{\tau\to\infty}{\longleftarrow}  \sum_{t=0}^\tau c\Pr[T > t] 
\; \stackrel{(\ref{eq:drift:eq2})}\le\;   \sum_{t=0}^\tau (\Exp[X_{t}] - \Exp[X_{t+1}])
\; = \; \Exp[X_0] - \underbrace{\Exp[X_{\tau+1}]}_{\geq 0}
\; \leq\;  n,
\end{align}
which proves the claim of the theorem. 
\end{proof}

\noindent
{\em Remark}. 
Note that in~\eqref{eq:drift:eq3}, some subsequence of $(\Exp[X_{\tau+1}])_{\tau\geq 0}$ converges to $\limsup_{t\to\infty} \Exp[X_t]$, which implies the slightly stronger statement $\Exp[T] \leq \tfrac 1c (n- \limsup_{t\to\infty} \Exp[X_t])$. 
This approach also allows for a reverse version of Theorem~\ref{thm:drift:t1}: if Condition~\eqref{eq:additivedrift} is replaced by the reverse condition $\Exp[X_{t+1}\mid X_t=x]\ge x- c$, then the same proof shows that $\Exp[T] \geq \tfrac 1c (n- \liminf_{t\to\infty} \Exp[X_t])$, i.e., then either $\Exp[T] = \infty$, or $\Exp[T]$ is a finite value which satisfies the inequality.\footnote{Note that $\Exp[X_t]$ does not need to converge to zero, even if $\Exp[T]$ is finite. For example, consider the Markov chain where $X_{t+1}$ is either $0$ or $2X_t$, both with probability $1/2$. Here $\Exp[T]=2$, but $\Exp[X_t]=X_0 =n$ for all $t\geq 0$. This example also shows that the ``naive'' reverse inequality $\Exp[T] \geq \tfrac nc$ does not hold in general, since the precondition $\Exp[X_{t+1}\mid X_t=x]\ge x- c$ is satisfied for all $c>0$.}\medskip

Note that the assumption of Theorem~\ref{thm:drift:t1} cannot hold for Markov chains whose state space $\S$ contains values arbitrarily close to zero. 
The following theorem allows to resolve this problem, as it shows that we can partition the analysis of such Markov chains into two phases: one coming 'close' to zero, the other actually reaching zero.

\begin{theorem}\label{thm:drift:t3}
Let $(X_t)_{t \in\N_0}$ be a Markov chain with countable state space ${\cal S} \subseteq [0,\infty)$ and assume $X_0=n$. Furthermore let $C>0$ be some positive constant and denote by $T_C$ the earliest point in time $t$ such
that $X_t \le C$. Assume that there exist constants $p_0 >0$ and $B>0$ such that for all $c\in {\cal S}$ with $0<c \leq C$ and for all $t\geq 0$ we have 
$$
 \Pr[X_{t+1}=0\mid X_t = c]  \ge p_0 
$$
and 
$$
\sum_{x\in {\cal S}, x > C}\Pr[X_{t+1}=x\mid X_t=c] \cdot \Exp[T_C \mid X_0=x] \le B.
$$
Then $\Exp[T] = \Exp[T_C] + O(1)$.
\end{theorem}


\begin{proof}
Let ${\cal S}_C := \{ s\in {\cal S}\mid s\le C\}$.
By the definition of $T_C$ the expectation of the number of steps until we reach a state $c\in {\cal S}_C$ is $\Exp[T_C]$. So assume now that we are in some nonzero state $c\in {\cal S}_C$. With probability at least $p_0$ we 
jump right into zero in the next step. Otherwise we  either end up in some (possibly different) state in  ${\cal S}_C$ or in some state $x > C$. In the latter case we need, due to our assumption, in expectation at most $B$ steps to return to a state in ${\cal S}_C$. We conclude that we have in expectation at least every $(B+1)$st step a chance of jumping to zero with probability at least $p_0$. As the expectation of a geometric random variable with success probability $p_0$ is $1/p_0$ this implies that, once we reach a state within ${\cal S}_C$ for the first time, the expectation for the number of additional steps to reach zero is bounded from above by $(B+1)/p_0$, as claimed.
\end{proof}

Consider a Markov chain on the non-negative integers which is defined as follows. From all states $x>0$  we move to one of the states $0,1,\ldots,x$ uniformly at random. How long does it take till we reach zero? Intuitively, we expect to half the distance towards zero in each step, so we expect that it will take $\log_2 n$ steps, if we start in $n$. What if we move uniformly to a state in $0,1,\ldots,ax$? For which $a>1$, if at all, do we have time $O(\log n)$ till we reach zero? We will give the (surprising) answer for this in the next section. The following theorem will turn out to be very useful. The clue is to consider a Markov chain in which the distance to zero is rescaled.

\begin{theorem}\label{thm:drift:t4}
Let $(X_t)_{t \in\N_0}$  be a Markov chain with state space $\S \subseteq \N_0$ and with $X_0=n$. Let $C\in\N_0$ be some positive constant and denote by $T_C$ the earliest point in time $t$ such
that $X_t \le C$. Assume furthermore that there exists a constant $c>0$ and a function $g: \S \rightarrow \R$ such that $g(x) = 0$ for all $x\le C$ and $g(x) > 0$ for all $x>C$ and such that for all $t\geq 0$
$$
\Exp[g(X_{t+1}) \mid  X_t= x]  \le g(x) - c\qquad\text{for all $x\in\N_0, x>C$.}
$$
Then
$\Exp[T_C] \leq g(n)/c$. 
\end{theorem}

\begin{proof}
Let $Y_t := g(X_t)$. Then the assumption of the theorem implies that for all $x>C$:
$$
\Exp[Y_{t+1}\mid Y_t = g(x)] = \sum_{x': g(x')=g(x)}\Pr[X_t =x' \mid Y_t = g(x)] \cdot \sum_{i>C} \Pr[X_{t+1}=i \mid X_t= x'] \cdot g(i) \le g(x) - c.
$$
We can thus apply Theorem~\ref{thm:drift:t1} to the Markov chain $(Y_t)$ to conclude that $\Exp[\widetilde{T}] \le g(n)/c$, where $\widetilde{T}$ is the time until the Markov chain $(Y_t)$ reaches zero. The observation that, by construction of the chain $(Y_t)$, we have  $\widetilde{T} \equiv T_C$  concludes the proof. 
\end{proof}

The next theorem is a simplified version of the previous one that is often easier to apply. The present formulation is from
Mitavskiy, Rowe, and Cannings~\cite{Mit-Row-Can:j:09}, the proof follows~\cite[Theorem~4.6]{Joh:th:10}.

\begin{theorem}\label{thm:drift:t2}
Let $(X_t)_{t \in\N_0}$ be a Markov chain with state space ${\cal S} \subseteq \{0\} \cup [1,\infty)$ and with $X_0=n$.  Let $T$ be the earliest point in time $t \geq 0$ such
that $X_t  = 0$. Suppose furthermore that there is a positive, increasing function $h: [1,\infty) \rightarrow \R_{> 0}$ such that for all $x \in \S$, $x> 0$ we have for all $t\geq 0$
$$
\Exp[X_{t+1}\mid X_t=x]\le  x-h(x).
$$
Then
$$
\Exp[T] \leq \frac{1}{h(1)}+\int_{1}^{n} \frac{1}{h(u)}du.
$$
\end{theorem}

\begin{proof}
Let 
\[
g(x) := \frac{1}{h(1)}+\int_{1}^{x} \frac{1}{h(u)}du \quad\text{for $x\in\S, x>0$}\qquad\text{and $g(0):=0$.}
\]

We claim that $g(x)-g(y) \geq \frac{x-y}{h(x)}$ for all $x,y\in \S$ such that $x >0$. To see this, observe that since $h$ is increasing and positive,
$$
g(x)-g(y) = 
\begin{cases}
\int_{y}^{x}  \frac{1}{h(u)}du \ge (x-y) \cdot  \frac{1}{h(x)}&\text{if } x > y >0,\\
\frac{1}{h(1)} + \int_{1}^{x}  \frac{1}{h(u)}du \ge   \frac{x}{h(x)}&\text{if } x > y =0,\\
0&\text{if } x = y > 0,\\
-\int_{x}^{y}  \frac{1}{h(u)}du \ge -(y-x) \cdot  \frac{1}{h(x)}&\text{if } 0 <x < y.
\end{cases}
$$

Now let us consider the random variable $Y_t := g(X_t)$. The above inequality for $g$ implies that for all $x\in \S \setminus\{0\}$ we have
\[
\Exp[g(X_t)-g(X_{t+1})\mid X_t=x] \geq \Exp\left[\frac{X_t-X_{t+1}}{h(X_t)}\mid X_t = x\right] \geq 1,
\]
where  the last inequality is just a reformulation of the assumption of the theorem.

Hence, we have $\Exp[Y_{t+1}\mid Y_t=g(x)] \le g(x)-1$, and the claim of the theorem thus follows by applying Theorem~\ref{thm:drift:t1} to $Y_t := g(X_t)$.
\end{proof}

A special case of Theorem~\ref{thm:drift:t2} is the case $h(x) = \delta x$, which is called \emph{multiplicative drift}; we then get $\Exp[T]\le \delta^{-1}(1+\log n)$, see also~\cite{Doe-Joh-Win:j:12:multiDrift}. In fact, a straightforward application of Markov's inequality even supplies us with exponential tail bounds, as was first noticed in~\cite{DoerrGoldberg:j:11:adaptiveDrift}.

\begin{theorem}\label{thm:multiplicative}
Let $(X_t)_{t \in\N_0}$ be a Markov chain with state space $\S \subseteq \{0\} \cup [1,\infty)$ and with $X_0=n$. Let $T$ be the earliest point in time $t \geq 0$ such
that $X_t  = 0$. Assume that there is $\delta>0$ such that for all $x \in \S$, $x> 0$ and for all $t\geq 0$ we have
\[
\Exp[X_{t+1}\mid X_t=x]\le  (1-\delta)x.
\]
Then
\[
\Pr\left[T > \left\lceil \frac{\log n +k}{|\log(1-\delta)|} \right\rceil\right] \leq e^{-k}.
\]
\end{theorem}
\begin{proof}
Let $t := \lceil (\log n +k)/|\log(1-\delta)| \rceil$. Since $\Exp[X_{t}] \leq (1-\delta)^t X_0 \leq e^{-k}$, the claim follows from Markov's inequality: $\Pr[T > t] = \Pr[X_t \geq 1] \leq e^{-k}$. Note that here we used $\S \subseteq \{0\} \cup [1,\infty)$; replacing the (somewhat arbitrary) boundary $1$ by some other constant would give a multiplicative constant in front of the probability $e^{-k}$.
\end{proof}

An important implication of Theorem~\ref{thm:multiplicative} is that if the probability of making large jumps is geometrically bounded, then also an additive drift suffices to get exponential tail bounds. Moreover, it is exponentially unlikely to cross an area with \emph{negative} drift in less than exponentially many steps. This idea again goes back to Hajek~\cite{Haj:j:82} and was first proven in this simplified form in~\cite{Oli-Wit:j:11:negativeDrift,oliveto2012erratum}. However, that proof is still based on some results from Hajek's paper. Here we give a short and fully self-contained proof.  

\begin{theorem}\label{thm:tailbounds}
For all $a,b,\delta, \eps, \gamma, r  >0$, with $a<b$, there is $c>0$, $n_0 \in \N$ such that the following holds.
Let $(X_t)_{t \in\N_0}$  be a Markov chain, and let $n\geq n_0$ such that for all $t \geq 0$ the following conditions hold:
\begin{enumerate} 
\item $\Exp[X_{t+1}-X_t \mid X_t =x ] \leq - \eps \quad \text{ for all $x>an$}$,
\item $\Pr[|X_{t+1}-X_t| \geq j] \leq r(1+\delta)^{-j}$ for all $j \in \N_0$.
\end{enumerate}
Let $T_a := \min\{t \geq 0 : X_t \leq an \mid X_0 = bn\}$ and $T_b := \min\{t \geq 0 : X_t \geq bn \mid X_0 = an\}$. Then 
\begin{enumerate}
\item[(a)] $\Pr[T_a \geq \frac{(1+\gamma)(b-a)n}{\eps}] \leq  e^{-cn}$.
\item[(b)] $\Pr[T_b \leq e^{cn}] \leq e^{-cn}$.
\end{enumerate}
\end{theorem}
\begin{proof}
We may assume $\gamma < 1/2$. Let $0<\eta <1$ be so small that $e^\eta < (1+\delta)^{1/2}$ holds. We claim that if $\eta$ is sufficiently small, then we can choose $j_0 >0$ such that $e^{\eta j} = 1+\eta j  \pm \gamma\eta\eps/6$ is satisfied for all $-j_0 \leq j \leq j_0$ (where we use $x=y\pm \eps$ as shorthand for $x\in [y-\eps,y+\eps]$), and such that at the same time
\begin{equation}\label{eq:tailofsum}
\sum_{j=j_0+1}^{\infty} r(1+\delta)^{-j/2} < \gamma\eta\eps/24 \qquad \text{ and } \qquad \sum_{j=j_0+1}^{\infty} jr(1+\delta)^{-j} < \gamma\eta \eps/24
\end{equation}
holds. Indeed, by setting (rather arbitrarily) $j_0 = \eta^{-1/3}$ we satisfy all three conditions if $\eta = \eta(\delta,\eps,\gamma,r)$ is small enough.


Consider $Y_t := e^{\eta X_t}$. For any $x > an$ we have
\begin{align*}
\Exp[Y_{t+1}-Y_{t} \mid X_t =x]  & = e^{\eta x}\Exp[e^{\eta (X_{t+1}-x)}-1 \mid X_t =x]\\
& = e^{\eta x} \bigg(\sum_{j=-j_0}^{j_0} \Pr[X_{t+1}-X_t = j \mid X_t = x]  \underbrace{(e^{\eta j}-1)}_{=\eta j  \pm \gamma\eta\eps/6} \pm\, 2\cdot \underbrace{\sum_{j=j_0+1}^{\infty}  r(1+\delta)^{-j} (e^{\eta j}+1)}_{\leq \gamma\eta\eps/12 \text{ by \eqref{eq:tailofsum}}} \bigg) \\
& = e^{\eta x} \bigg(\sum_{j=-j_0}^{j_0} \Pr[X_{t+1}-X_t = j \mid X_t = x]\eta j \pm  \gamma\eta\eps/3 \bigg) \\
& \stackrel{\eqref{eq:tailofsum}, \text{Cond.2}}{=} e^{\eta x}(\Exp[X_{t+1}-X_t \mid X_t = x]\eta \pm \gamma\eta\eps/2)  <-e^{\eta x}\underbrace{\eta\eps(1-\gamma/2)}_{=:\nu}.
\end{align*}
Recalling that $Y_t = e^{\eta x}$ for $X_t =x$, we thus find that $Y_t$ has a multiplicative drift with factor $1-\nu \leq 1-\tfrac98\eta\eps/(1+\gamma)$. The first statement now follows by applying Theorem~\ref{thm:multiplicative} to $Y_t$. For the second statement, consider the event $\mathcal{E}$ that for the first $e^{cn}$ points in time where $X_t \leq an$ we have $X_{t+1} \leq X_t+ (b-a)n/2$. Note that $\Pr[\neg \mathcal{E}] \le e^{cn}\cdot r(1-\delta)^{n(b-a)/2} = e^{-\Omega(n)}$, if we choose $c>0$ sufficiently small. From now on we may thus condition on $\mathcal{E}$. Consider some time $t_1$ at which we cross the lower bound from below, i.e., $X_{t_1-1} \leq an$ and $X_{t_1} > an$. Let $t_2$ be the next point in time where we drop back below $an$, so $t_2 = \min\{t > t_1 \mid X_{t_2} \leq an\}$. Note that the event $\mathcal{E}$ implies that for all $t_1 < e^{cn}$ we have $X_{t_1} < n(a+b)/2$. We now show that it is very unlikely that we reach $b$ between $t_1$ and $\max\{t_2,e^{cn}\}$. In fact, this can only happen if $Y_t \geq e^{\eta bn}$. Thus, using the multiplicative drift of $Y_t$ and Markov's inequality, for all $t_1 < t < \max\{t_2,e^{cn}\}$ and $an< x \le (a+b)n/2$,
\[
\Pr[X_t > bn \mid X_{t_1} = x ] \leq \frac{\Exp[Y_t \mid Y_{t_1} = e^{\eta x}]}{e^{\eta bn}} \leq \frac{(1-\gamma)^{t-t_1}e^{\eta y}}{e^{\eta bn}} \leq (1-\gamma)^{t-t_1}e^{-\eta (b-a)n/2}.
\]
A union bound over all $t_1< t < \max\{t_2,e^{cn}\}$ shows
that (conditioned on $\mathcal{E}$) the probability that $bn$ is reached in the interval $[t_1,t_2]$ is $e^{-\Omega(n)}$. A second union bound over all $t_1 < e^{cn}$ thus shows the second claim for $c$  sufficiently small. 
\end{proof}

\section{Example: Random Decline}


With Theorems~\ref{thm:drift:t2} and~\ref{thm:multiplicative} at hand we can now analyse the Markov chain that, from state $x$, moves to a state uniformly at random within $\{0,\ldots,\lfloor ax\rfloor \}$. Clearly, $\Exp[X_{t+1}\mid X_t=x] \leq {\textstyle \frac{a}2 x}$
and we can apply Theorem~\ref{thm:multiplicative} to deduce that the number of steps $T$ until we reach zero is $O(\log n)$ for all $a<2$. For $a>2$ it may seem that    $\Exp[T] = \infty$, as we have a drift to the right. Perhaps surprisingly, this is not true. In fact, we have 
the following theorem.
\begin{theorem}\label{thm:randomdecline}
For the Markov chain defined above with (constant) parameter $a$, we have $\Exp[T] = O(\log n)$ if and only if $a < e$, where $e=2.718..$ is the Euler constant.
\end{theorem}
\begin{proof}
Assume $a<e$. Let $g(x) = \log(x)$ for $x>C$  (where $C$ is some constant depending on $a$ that we will fix below)  and $g(x) = 0$ for all $x\le C$. Then for all $x>C$,
\begin{align*}
\sum_{m> C} \Pr[X_{t+1}=m \mid X_t= x] \cdot g(m)  
&= \frac1{\lfloor ax\rfloor+1} \sum_{m=C+1}^{\lfloor ax\rfloor} \log(m) \leq \log(x) -1 + \log(a) + \frac1{\lfloor aC\rfloor+1}.
\end{align*}
As $\log(a) < 1$ for all $a< e$ we can thus find for each  $a<e$ a constant $C$ so that we can apply Theorem~\ref{thm:drift:t4} to deduce that  $\Exp[T_C]\le O(\log n)$.
Theorem~\ref{thm:drift:t3} then implies that we also have $\Exp[T] = O(\log n)$. Indeed, the assumptions of  Theorem~\ref{thm:drift:t3} are easily seen to be fulfilled for $p_0 = 1/(C+1)$ and $B = \max_{C+1\le n \le aC} \Exp[T_C\mid X_0=n]$ (which is finite, as we maximize over a finite number of expectations and we have already shown that these expectations are finite).

For the other direction, we restrict ourselves for concreteness to $a=e$. (The case $a>e$ stochastically dominates the case $a=e$, and thus the statement for $a>e$ follows easily by a coupling argument.) 
For the sake of contradiction we assume that $\Exp[T] = O(\log n)$, so assume that there is $D>0$ such that $\Exp[T] \leq D \log n$ for all $n$. This implies in particular that $\Exp[T]$ is finite. Note that we may assume $D \geq 2$, and that $D$ is an integer. Let $g(x) = 2D\log(x/D^2)$ for $x>D^2$, $g(x) = 1$ for $1 \leq x \leq D^2$, and $g(x)=0$ for $x=0$. We set $Y_t := g(X_t)$, where $X_t$ is the state in the $t$-th step. Note that $X_t =0$ if and only if $Y_t=0$. We claim that the reverse of Condition~\eqref{eq:additivedrift} in Theorem~\ref{thm:drift:t1} is satisfied for $Y_t$ for $c=1$. To see this consider some $y >0$ of the form $g(x)$, where $x \in \N^+$.  We want to show that $\Exp[Y_{t+1}\mid Y_t= y] \geq y-1$. For $x \leq D^2$ this is trivial since then $y = g(x) \leq 1$. So assume that $x > D^2$. We compute similarly as above (using that $s\cdot \log(s/(eD^2)$ is an antiderivative of $\log(s/D^2)$ and that $\lfloor ex \rfloor \geq ex-1$ implies $\log(\lfloor ex \rfloor/(eD^2)) \geq \log (x/D^2) - 2/(ex)$) that
\begin{align*}
\Exp[Y_{t+1}\mid Y_t= y]
& = \Exp[g(X_{t+1})\mid X_t= x]  
\geq \frac1{\lfloor ex\rfloor +1} \sum_{m=D^2}^{\lfloor ex\rfloor} (2D\log(m/D^2))\; \\
 &\ge \frac{\lfloor ex\rfloor}{\lfloor ex\rfloor+1}\left(2D\log\left(\frac{\lfloor ex \rfloor}{eD^2}\right)+\frac{2D^3}{\lfloor ex\rfloor} \right)\\
 & \geq 2D\log(x/D^2) -1 \;=\;y-1.
\end{align*}
Since this inequality is true for all $y = g(x)>0$, $x\in \N^+$, we may use the reverse version of Theorem~\ref{thm:drift:t1}, cf.\!  the remark after the proof of Theorem~\ref{thm:drift:t1}. 
Hence, for $n > D^2$,
\[
\Exp[T] \geq Y_0- \liminf_{t\to\infty} \Exp[Y_t] = 2D\log(n/D^2)- \liminf_{t\to\infty} \Exp[Y_t].
\]
Since we assumed $\Exp[T] \leq D \log n$, the last inequality implies for all sufficiently large $n$ and $t$
\begin{equation*}
\Exp[Y_t] \geq D\log(n)/2.
\end{equation*}
Let $p_t := \Pr[Y_t >0] = \Pr[T > t]$, and let $E_t := \Exp[Y_t \mid Y_t >0]$. Then $p_tE_t = \Exp[Y_t] \geq D\log(n)/2$ for sufficiently large $n$ and $t$. On the other hand, since $X_t$ can increase in each step by at most a factor of $e$, we can bound $X_t \leq ne^t$, and thus $E_t \leq 2D\log(ne^t/D^2)$. Combining both bounds, we obtain
\[
p_t \geq \frac{D\log n}{2 E_t} \geq \frac{D\log n}{4D\log(ne^t/D^2)} = \frac{D\log n}{4Dt +4D\log(n/D^2)}.
\]
However, this implies that
\[
\Exp[T] = \sum_{t\geq 0} p_t = \infty,
\] 
contradicting that $\Exp[T]$ is finite. Thus the assumption $\Exp[T] = O(\log n)$  for $a=e$ was false.
\end{proof}

\section{Evolutionary Algorithm}\label{sec:EA}

In this section we (re)prove some fundamental properties of the following classical algorithm~\cite{back1996evolutionary} for optimizing a \emph{fitness function} $f:\{0,1\}^n \to \R$, which we call \emph{evolutionary algorithm} or \emph{EA}\footnote{In the evolutionary algorithms community it is called (1+1)-Evolutionary Algorithm.}. The algorithm starts with a 
search point $\xi^0 \in \{0,1\}^n$  that we usually assume to be chosen uniformly at random. It then proceeds in rounds. In each round it flips each bit of the current search point $\xi$ with probability (or \emph{mutation rate}) $c/n$, where $c>0$ is the \emph{mutation parameter}. The thus created \emph{offspring} $\xi'$ replaces $\xi$ if and only if $f(\xi') \geq f(\xi)$. We are interested in the \emph{optimization time} of $f$, i.e., in the number of rounds until a (global) maximum of $f$ is found. Note that we may assume that the algorithm has access to the function~$f$ only via an oracle. That is, the algorithm can query the function value $f(\xi)$ for a given $\xi$, but has no further knowledge about the function. In particular, the algorithm thus has no chance to `know' which $x$ maximizes the function before having queried all $2^n$ values. This is the reason for the above definition of the optimization time: we count the number of rounds until the algorithm first queries a global maximum. From then on the algorithm may continue asking queries, but clearly, the function value will never change again.  

Throughout this section, we will use the following notation. For all $t \geq 0$, we denote by $\xi^{t} \in \{0,1\}^n$ the search point after $t$ rounds of the algorithm. For any non-empty set of indices $I \subseteq [n]$, let $d(I,t) := |\{i \in I \mid \xi_i^{t} = 0\}|/ |I|$ be the density of zero bits in the $I$-substring of $\xi^{t}$. We write $d(I)$ instead of $d(I,t)$ if $t$ is clear from the context. 

By the \emph{no free lunch} principle, for arbitrary functions $f$ there is no hope that the algorithm will be fast. For example, for the function $f(\xi)$ that is one for $\xi=x_0$ and zero otherwise, \emph{any} algorithm will have exponential optimization time whp\footnote{\emph{with high probability}, i.e., with probability tending to one as $n\to \infty$}. Nevertheless algorithm EA gives good results in many practical applications, indicating the algorithm should be faster for a suitably restricted class of functions. In this section, we will study two classes of functions: \emph{strictly monotone functions}, that is, functions $f$ that satisfy $f(\xi) > f(\xi')$ for all $\xi,\xi'$ so that $\xi\not=\xi'$ and $\xi_i\ge \xi_i'$ for all $1\le i\le n$; and \emph{linear functions}, i.e., functions of the form $f(\xi) = \sum_{i=1}^n a_i\xi_i$ with weights $a_1,\ldots,a_n \in \R$. To avoid trivialities, we will always assume that all $a_i$ are non-zero. Moreover, by symmetry of the EA we may even assume that the weights are sorted and positive, $a_1\geq \ldots \geq a_n >0$. Note that with these assumptions every linear function is in particular strictly monotone.

For strictly monotone functions it easily follows from the drift theorems in Section~\ref{drift:theorems} that the optimization time of algorithm EA is $\Theta(n\log n)$ for all mutation parameters $c <1$, see also~\cite{doerr2013mutation}.

\begin{theorem}\label{thm:clessthan1}
For every constant $c<1$ and all strictly monotone functions $f$, the optimization time of  algorithm EA with mutation rate $c/n$ is $\Theta(n\log n)$ in expectation and whp.
\end{theorem}
\begin{proof}
For the lower bound, simply observe that whp the initial string has at least $n/3$ zero bits, and that each of these bits needs to be flipped at least once to reach the optimum. Therefore, we need to flip at least $\Omega(n\log n)$ bits in expectation and whp (as this is essentially a coupon collector process). Thus we also need $\Omega(n\log n)$ rounds to achieve so many flips. We omit the details.

For the upper bound, let $X_{t} := |\{i \in [n]: \xi_i^{t}=0\}|$ be the number of zero bits in the $t$-th step of the algorithm. We want to bound $\Exp[X_{t+1}\mid X_{t} = x]$. Assume we flip $r$ zero bits to a one and $s$ one bits to a zero. If $r=0$, then we reject the offspring~$\xi'$ and thus $X_{t+1} = X_{t}$.
If $r>0$ and $s=0$ then we accept $\xi'$ and thus $X_{t+1} \le X_{t} -1$, with room to spare. For the remaining case $r>0$ and $s\ge 1$ we use the bound $X_{t+1} \le X_{t} +s -1$, which is true regardless whether we accept or reject~$\xi'$. 

Denote by $p_{r,0}$ the probability that $r=0$. Then the above observations imply
\begin{align*}
\Exp[X_{t+1}\mid X_{t} = x]&\le p_{r,0}x + (1-p_{r,0})(x-1+\Exp[s]) = x -  (1-p_{r,0})(1-\Exp[s] ).
\end{align*}
Note that $p_{r,0}$ is the probability that a binomially distributed random variable with parameters $x$ and $p=c/n$ is equal to zero. Here the following inequality about the binomial distribution is helpful:
\begin{equation}\label{eq:drift:binom}
\Pr[\Bin(n,p) > 0] \ge \frac{np}{1+np}\qquad\text{for all $n\in\N_0$ and $0\le p\le 1$}.
\end{equation}
(Easy to check by induction on $n$.) With $(\ref{eq:drift:binom})$ and the observation that $\Exp[s] = (n-x)\cdot (c/n)\le c$ we deduce
\begin{align*}
\Exp[X_{t+1}\mid X_{t} = x]&\le x - \frac{xc/n}{1+xc/n}(1-c) \le x - \frac{xc(1-c)}{n(1+c)}
\end{align*}
and we can apply Theorem~\ref{thm:drift:t2} with $h(x)= xc(1-c)/(n(1+c))$ in order to deduce that $\Exp[T] = O(n\log n)$. Moreover, by Theorem~\ref{thm:multiplicative} (with $\delta = c(1-c)/(n(1+c))$) the same bound holds whp.
\end{proof}

We will see in the next section that the restriction $c<1$ is essential: for $c\ge 2.2$ the running time is not $O(n\log n)$, but actually exponential in $n$. Before we proceed to that proof we state a lemma on linear functions that will turn out to be very useful in Section~\ref{sec:monotone} as well as Section~\ref{sec:linear}. 

%

\begin{lemma}\label{lem:lindensity}
Consider the algorithm EA on a linear  function, $f(x) = \sum_{i=1}^n a_ix_i$, and assume that $a_i \geq a_j$ for $i\leq j$. Let $\alpha,\beta,\eps > 0$ be any constants and let $I,J\subseteq [n]$ be two sets of indices such that $\max I < \min J$ and such that $|I| = \alpha n$ and $|J| =\beta n$. Then there exists $C > 0$ such that for any initial value $\xi^0$ of the algorithm, whp we have $d(I, t) \le d(J, t) + \eps$  for all $t \in [Cn, e^{n/C}]$. The interval may be replaced by $[0,e^{n/C}]$ if 
$d(I,0) \le d(J,0) + \eps/2$.
\end{lemma}
\begin{proof}
We use a coupling that was first applied in~\cite{jagerskupper2011combining}. We first consider a single step, so fix $t \geq 0$ and $\xi^t$. Let $i< j$ be two positions such that $\xi_i^{t}=\xi_j^{t}=0$. We claim that $\Pr[\xi_i^{t+1} = 1] \ge \Pr[\xi_j^{t+1}=1]$. To see this assume first that all bit flips except the $i$th and the $j$th are fixed. Then the following holds. If the bit flips for the $i$th and the $j$th bit are identical then we change either none of the two bits or both. Thus, the only interesting case is if exactly one of these two bits is flipped. Here we observe the following: if the case that the $j$th bit is flipped is accepted then so is the case that the $i$th bit is flipped (due to the fact that we assumed $a_1 \ge a_2\ge \ldots\ge a_n$). A few moments of thought show that this already proves our claim. 

Similarly we deduce that if $k< \ell$ are two positions such that $\xi_k^{t}=\xi_\ell^{t}=1$, then we have $\Pr[\xi_k^{t+1} = 0] \le \Pr[\xi_\ell^{t+1}=0]$. 

Now we can prove the lemma.
We claim that by Theorem~\ref{thm:tailbounds} it suffices to show that $n\cdot (d(I)-d(J))$ has a negative drift, whenever $d(I)\ge d(J)+\eps/2$. Indeed, then we may first apply Theorem~\ref{thm:tailbounds} (a) to $X_t := n\cdot (d(I,t)-d(J,t))$ with $a:= \eps/2$ and $b:= d(I,0)-d(J,0)$ to conclude that after a linear number of steps, whp $X_t \leq \eps n/2 $. Afterwards, we may apply Theorem~\ref{thm:tailbounds} (b) again with the same $X_t$ and $a$, but with $b:= \eps$ to conclude that whp $X_t \leq \eps n$ for an exponential number of steps.

So let us show that $n\cdot (d(I)-d(J))$ has a negative drift for $d(I)\ge d(J)+\eps/2$. So assume that $d(I,t)\ge d(J,t)+\eps/2$ holds for some $t \geq 0$. Let $p_{I0} := \min\{\Pr[\xi_i^{t+1}=1]:i\in I, \xi_i^{t}=0\}$ and $p_{J0} := \max\{\Pr[\xi_j^{t+1}=1]:j\in J, \xi_j^{t}=0\}$ and observe that we know from above that $p_{I0} \ge p_{J0}$. Similarly, we have for $p_{I1} := \max\{\Pr[\xi_k^{t+1}=0]:k\in I, \xi_k^{t}=1\}$ and $p_{J1} := \min\{\Pr[\xi_\ell^{t+1}=0]:\ell\in J, \xi_\ell^{t}=1\}$ that $p_{I1} \le p_{J1}$. Moreover, observe that $p_{J0} = \Omega(1/n)$, because the probability to flip the $\ell$-th bit and no other bit is $\Omega(1/n)$. Thus, the
the expected change of $n\cdot (d(I,t)-d(J,t))$ is at most
\begin{align*}
n[(1-d(I,t)) p_{I1} - d(I,t)  p_{I0} ] - n[[(1-d(J,t)) p_{J1} - d(J,t) p_{J0} ] \le 
n \underbrace{(d(J,t)-d(I,t))}_{\le -\eps/2}  (p_{I1}+ p_{J0}),
\end{align*}
which is $-\Omega(1)$, as desired. 
\end{proof}

\subsection{Monotone Functions}\label{sec:monotone}

We already saw that for $c<1$ the runtime of algorithm EA is $\Theta(n\log n)$ for all strictly monotone functions. 
For the natural choice $c=1$ it is still unknown whether the runtime is $\Theta(n\log n)$. 
The best known upper bound is $O(n^{3/2})$ due to Jansen~\cite{jansen2007brittleness}. 
The following theorem explains why this may be not so easy to show, as we do get an exponential optimization time for $c$'s that are not much larger. 

\begin{theorem}
For every constant $c\geq 2.2$, there is a strictly monotone function such that whp the optimization time of the EA with mutation rate $c/n$ is $e^{\Omega(n)}$.
\end{theorem}

\noindent
A similar theorem was shown by Doerr et al.~\cite{doerr2013mutation} with $c>16$ instead of $c\ge 2.2$.

\begin{proof}
The main idea of our proof is the following. Consider $n$ bits and split them in high order bits (first $\alpha n$, say) and low order bits (the remaining $(1-\alpha)n$). Now consider the (linear) fitness function $a_i = n$ for $1\le i\le \alpha n$ and $a_i = 1$ otherwise. Then the density of zero bits will decrease faster among the high order bits. In particular, there will be a time when the density among the high order bits is $\eps$, but the density among the low order bits is still $\eps+\delta$, for some $\delta = \delta(\alpha, \eps, c) >0$. The overall density in the string is then $\eps + (1-\alpha)\delta > \eps$.

If we now change the fitness function by picking a random set of $\alpha n$ bits as 'new' high order bits, we will start with a density of $\eps + (1-\alpha)\delta$ among the high order bits and will decrease it during the second round to density $\eps$. Whenever we decrease the number of zero bits among the high order bits, we accept the offspring regardless of the changes of low order bits. In particular, in such rounds the density among the low order bits will \emph{increase} in expectation towards $1/2$. By choosing the constants appropriately, this yields overall a \emph{positive} drift for the density among the low order bits. The drift is actually strong enough that the density will increase from $\eps + (1-\alpha)\delta$ to $\eps +\delta$ in a very short time -- this will finish before the density among high order bits can decrease from $\eps + (1-\alpha)\delta$ to $\eps$. So after the second round again we are in the situation that the density among high order bits is $\eps$ and the density among low order bits is (at least) $\eps +\delta$. We will show that we can play this game for exponentially many rounds.

Formally we proceed as follows. For a given $c\ge 2.2$ we choose $0< \alpha < 1/2$ which satisfies
 \begin{equation}\label{eq:cond}
 \alpha c - e^{-(1-\alpha) c}   > \frac{\alpha}{1-\alpha}.
 \end{equation}
A numerical calculation shows that this is possible for all $c \geq 2.2$. We choose a constant $\beta = \beta(\alpha,c)$ which is sufficiently small compared to $\alpha$ (by abuse of notation we denote this by $\alpha \gg \beta$). To provide an overview, the constants in the proof will satisfy 
\begin{equation}\label{eq:constantsmonotone}
1/2 > \alpha \gg \eps \gg \delta \gg \beta \gg \gamma \gg \mu >0,
\end{equation}
where $\gamma$ and $\mu$ are constants that will appear later in the proof. The $\gg$-sign hereby indicates the only restriction that we require: the constants are chosen from left to right, so that $\alpha$ satisfies equation~\eqref{eq:cond} and then every constant is suitably small compared to all previously chosen constants.

For $1 \leq i \leq e^{\mu n}+1$ we choose sets $A_i \subseteq [n]$ of size $\alpha n$ independently and uniformly at random, and we choose sets $B_i\subseteq A_i$ uniformly at random of size $\beta n$. We define the {\em level} $\ell(\xi)$ of a vector $\xi\in\{0,1\}^n$ by 
\begin{equation}\label{as:monotone:e1}
\ell(\xi) := \max \{ \ell' \in [e^{\mu n}] : |\{j \in B_{\ell'} : \xi_j = 0\} |\le \eps |B_{\ell'}|\}\quad\text{(with $\ell(\xi)=0$, if no such $\ell'$ exists).}
\end{equation}

With this notation at hand we can then define a fitness function $f:  \xi\in\{0,1\}^n \to \R$ as follows:
\begin{equation}\label{as:monotone:e2}
f(\xi) =\ell(\xi) \cdot n^{2} +  \sum_{i\in A_{\ell(\xi)+1}}\xi_i\cdot n +  \sum_{i\not\in A_{\ell(\xi)+1}} \xi_i. 
\end{equation}

So the set $A_{i+1}$ describes the set of high order bits at level $i$, where the level is determined by the sets $B_i$. One easily checks that this function is monotone. Indeed, assume that $\xi$ is dominated by $\xi'$, i.e., $\xi_i \leq \xi_i'$ for all $1\leq i \leq n$. Then $\ell(\xi) \leq \ell(\xi')$. If $\ell(\xi) = \ell(\xi')$ then it is obvious that $f(\xi)\leq f(\xi')$ (and we have a strict inequality if $\xi \neq \xi'$). On the other hand, if $\ell(\xi) < \ell(\xi')$ then $f(\xi) < \ell(\xi)n^{2} + n^2  \leq \ell(\xi')n^{2}  \le f(\xi')$, as desired.

Observe also that, by definition of the function~$f$, the algorithm will never accept a change that will decrease the \emph{current level} $\ell(t) := \ell(\xi^{t})$. That is, when the current level is $i$ the density of zero bits will stay below $\eps$ within $B_i$ (or the level increases). 


Note that the definition of $\ell$ refers to {\em all} sets $B_i$, which makes the analysis quite tricky.  In order to avoid this we define a slightly different level function $\tilde\ell = \tilde\ell(\xi,t)$ that depends on the time~$t$. More precisely, recall that $\xi^t$ denotes the search point after round $t$. In order to compute $\tilde\ell(\xi,t+1)$ for a search point $\xi$ we proceed as follows: instead of considering the maximum over all integers $\ell' \in [e^{\mu n}]$ as in \eqref{as:monotone:e1}, we  only consider values $\ell' \in[ \tilde\ell(\xi^t, t) +1]$. Informally, $\tilde\ell$ thus coincides with $\ell$ except that we only accept level gains of exactly one. For brevity, we write $\tilde\ell(t) := \tilde\ell(\xi^t,t)$ to denote the value of $\tilde\ell$ at time $t$. We note, however, that $\tilde\ell(t)$ implicitly depends on the whole history of the algorithm. We now  consider algorithm EA on the time-dependent fitness function
$$
\tilde f(\xi,t) =\tilde \ell(t) \cdot n^{2} +  \sum_{i\in A_{\tilde \ell(t)+1}}\xi_i\cdot n +  \sum_{i\not\in A_{\tilde \ell(t)+1}} \xi_i.
$$
As long as $\tilde \ell(t)$ coincides with $\ell(t)$, we have $\tilde f(\xi^t,t) = f(\xi^t)$. The algorithm thus performs the same steps on both functions as long as $\tilde \ell(t) = \ell(t)$. We will show later that whp $\tilde\ell(t)$ and $\ell(t)$ indeed coincide for an exponential number of rounds. The advantage of $\tilde f$ is that we do not have to choose the sets $A_i, B_i$ in advance: we can postpone choosing $A_{i+1}, B_{i+1}$ until we enter the $i$th level. 

Before we now dive into the proof we fix some notations. Let $T := e^{\mu n}$. Our aim is to show that whp the algorithm will run for at least $T$ rounds. With $t_i$ we denote the points in time where $\tilde \ell$ increases from $i-1$ to $i$, i.e., $\tilde \ell(t_{i}-1)=i-1$ and $\tilde \ell(t_{i})=i$. Moreover, for convenience we set $t_0 :=0$ and $A_0 := B_0 := B_{-1} = \emptyset$. Recall that for a set $S\subseteq [n]$ we denote 
by $d(S,t)$ the density of zero bits in the set $S$ with respect to the $t$-th search point $\xi^{t}$. 

Observe that algorithm EA (with respect to the function $\tilde f$) does not need to know $A_{i+1}$ before time $t_i$. By Chernoff bounds (and a union bound over all $1\le t\le t_i \le T$) we can thus assume (by making $\mu$ sufficiently small) that for an arbitrarily small constant $\gamma$ we have whp that
\begin{equation}\label{eq:lateAi}
\text{for all $i$ s.t. $t_i\le T$:}\quad d(A_{i+1},t) =  d([n],t)\pm\gamma \quad\text{for all }1\le t\le t_i.
\end{equation}
From time $t_i$ onwards algorithm EA now optimizes with respect to the set $A_{i+1}$. Note that we need to know the set $B_{i+1}$ only in order to check whether we should increase $\tilde\ell$. If we thus consider a variant of the process in which we only check whether we reject an offspring or not based on the sets $B_1,\ldots,B_i$  (i.e., we do not check whether we should increase the level), then we get a sequence of search points $\hat\xi^t$ that coincide with $\xi^t$ for all $t < t_{i+1}$. On the other hand, Chernoff bounds and a union bound over all $1\le t\le T$ imply that for this modified sequence a randomly chosen set $B\subseteq A_{i+1}$ of size $|B|=\beta n$ satisfies $d(B) = d(A_{i+1},t)\pm\gamma$ for all $t< T$. In particular we thus have
\begin{equation}\label{eq:lateBi}
\text{for all $i$ s.t. $t_{i}\le T$ and all $1\le t < \min\{t_{i+1},T\}$:}\quad d(B_{i+1},t) =  d(A_{i+1},t)\pm\gamma.
\end{equation}
Finally, the probability to flip a linear number of bits in one round is exponentially small, so whp we have 
\begin{equation}\label{eq:smallsteps}
\text{for all $1 \leq t \leq T$: \quad the number of bits flipped in round $t$ is less than $\beta\gamma n$.} 
\end{equation}
This condition implies in particular that $d(A_{i},t+1) =  d(A_{i},t)\pm\gamma$ and $d(B_{i},t+1) =  d(B_{i},t)\pm\gamma$ for all $i\geq 1$, $t\leq T$. Also note that the probability to flip at least $k$ bits in one round drops at least geometrically as $k \to \infty$, so the number of zero bits in $\xi^t$ satisfies condition 2 in Theorem~\ref{thm:tailbounds}. Similarly for the number of zero bits within $A_i$ or within $B_i$ in $\xi^t$. Note that this implies in particular, that we may assume that within the first $T$ steps of algorithm EA the density of any of these sets will never change opposite to the direction of its drift by more than~$\gamma$. 

Assume now that we can show (we will do this below, see (iii)) that algorithm EA with respect to the function $\tilde f$ satisfies 
$d([n],t) \ge \eps + 3\gamma$ for all $1\le t\le T$. Then one easily checks that~\eqref{eq:lateAi} and~\eqref{eq:lateBi} imply that for all $i$ such that $t_i\le T$ we have $d(B_{i},t) > \eps$ for all $t<t_i$  which implies that $\ell(t) = \tilde\ell(t)$ for all $t\le T$. Thus it indeed suffices to analyse EA with respect to $\tilde f$, and this is what we will now do.
More precisely, for a suitably chosen constant $\delta >0$ we will show by induction that for every fixed 
$i$ such that $t_i\le T$ the following three properties hold with probability at least $1-e^{-2\mu n}$: 
\begin{enumerate}[(i)]
\item $d(A_{i},t_{i})= \eps \pm 2\gamma$,
\item $d([n] \setminus (A_{i}\cup B_{i-1}), t_i) \ge \eps +\delta$,
\item $\forall t_{i} \le t < \min\{t_{i+1},T\}$:\ \ $d([n],t)\ge \eps+3\gamma$.
\end{enumerate}
Note that the choice of $T$ together with a union bound argument thus implies that these properties all hold for all $i$ simultaneously, which then concludes the proof.

First consider (i). We  know by~\eqref{eq:lateBi} that $d(A_i, t_i-1) = d(B_i, t_i-1)\pm \gamma$. For $i>1$, by induction hypothesis (iii) holds for all $t < t_i$, and we have argued already that this implies $d(B_i, t)> \eps$ for all $t<t_i$. Note that the latter is also true for the base case $i=1$. As $d(B_i, t_{i}) \le \eps$ by definition of $t_i$, the claim thus follows easily together with~\eqref{eq:smallsteps}.
 
Next we prove (ii) and (iii) together by induction. More precisely, we will show that if (ii) holds for $i-1$ then (iii) also holds for $i-1$ and (ii) also hold for $i$. Since $t_{0}=0$ and $A_0 = B_{-1}=\emptyset$, and recalling that $\xi^0$ is chosen randomly, we have that (ii) holds for $i=0$, which gives us the start of the induction. So consider some $i\ge 1$.
By the induction assumption we know that
\begin{equation}\label{eq:totaldensity}
d([n],t_{i-1})\stackrel{{\scriptsize (i),(ii)}}\ge \alpha(\eps-2\gamma)+(1-\alpha-\beta)(\eps+\delta)  \ge \eps + (1-\alpha)\delta-\beta
\end{equation} by suitable choice of the constants, cf.~\eqref{eq:constantsmonotone}. By definition of $t_{i-1}$  and~\eqref{eq:smallsteps} we know that 
\begin{equation}\label{eq:monotone:hlp:e1}
\eps-\gamma \le d(B_{i-1},t_{i-1})\le \eps. 
\end{equation} 
Let $R=[n] \setminus (A_{i}\cup B_{i-1})$. By the (random) choice of $A_i$ and~\eqref{eq:lateAi} we thus know that
\begin{equation}\label{eq:worstcasestart}
d(A_{i}, t_{i-1}) = d([n],t_{i-1}) \pm \gamma\qquad\text{and}\qquad d(R,t_{i-1}) = d([n]\setminus B_{i-1},t_{i-1}) \pm \gamma \stackrel{\eqref{eq:totaldensity},\eqref{eq:monotone:hlp:e1}}{\ge}
d([n],t_{i-1})\pm \gamma. 
\end{equation}

We first compute the drift within $A_{i}$ while $d(A_i,t) > \eps + \gamma$ (and in particular, $t_{i-1} \leq t < t_i$ by \eqref{eq:lateBi}). 
To do this observe that the choice of the weights in the function~$\tilde f$ are such that the bit flips within  $A_{i}$ dominate: if they lead to a decreased number of zeroes within $A_{i}$ then the offspring is accepted regardless of what happens for the bits outside of $A_{i}$ (unless the density within $B_{i-1}$ grows above $\eps$ and the offspring is thus rejected for this reason). Similarly, if the number of zeroes within $A_{i}$ increases then the offspring is always rejected (note that the assumption $d(A_i,t) > \eps + \gamma$ together with $(\ref{eq:lateBi})$ implies that $\tilde \ell$ cannot increase). 
In particular, the density within $A_{i}$ is thus non-increasing. 

Consider $t_i' :=\inf \{t \in [t_{i-1},T]: d(A_{i},t) \leq \eps+(1-\alpha)\delta-\beta\}$, and note that $t_i' \leq t_i$ by \eqref{eq:lateBi}. Then $d(A_{i},t_i') \geq \eps + (1-\alpha)\delta-\beta-\gamma$: for $t_i'> t_{i-1}$ this follows from the definition of $t_i'$ and~\eqref{eq:smallsteps}, while for $t_i' = t_{i-1}$ it follows from~\eqref{eq:totaldensity} and~\eqref{eq:worstcasestart}.
Since the density in $A_{i}$ is non-increasing whenever $d(A_i,t) > \eps + \gamma$, we deduce that for all $t\in [t_i',\min\{t_i,T\}]$ the probability that we flip at least two zero bits in $A_{i}$ to a one is bounded by $O(\eps^2)$. The drift within $A_{i}$ is thus essentially determined by the probability that we flip exactly one bit within $A_{i}$ (which is $(1+o(1))\alpha c e^{-\alpha c}$) and that this bit is a zero bit (which gives a factor $d(A_{i},t)$). Note also that the probability that we flip a bit within $B_{i-1}$ is bounded by $c\beta$, which we may assume to be also of the order $O(\eps^2)$ by~\eqref{eq:constantsmonotone}. We thus know that the expected 
drift for the density within $A_{i}$  is  for all $t_i'\le t\le \min\{t_i,T\}$
$$
-\frac{ (1+o(1))\alpha c e^{-\alpha c}\cdot d(A_{i},t)  +O(\eps^2)}{\alpha n}= -\frac{\eps c}n e^{-\alpha c} (1+O(\eps)) =: \Delta_1.
$$

Next we consider what happens within $R$. By \eqref{eq:worstcasestart} and Lemma~\ref{lem:lindensity} we may assume that 
$d(R,t) \ge d(A_{i},t)-3\gamma$ for all  $t_{i-1}\le t\le \min\{t_i,T\}$. In particular, $d(R,t_i') \ge d(A_{i},t_i') - 3\gamma \geq \eps+(1-\alpha)\delta - \beta - 4 \gamma$. We want to show that the density $d(R)$ {\em increases} to at least $\eps + \delta$. Thus, we may assume without loss of generality that
$d(R,t) = \eps \pm 2\delta$ for all $t_{i}'\le t< \min\{t_i,T\}$.
We will now compute a lower bound on the drift within $R$. Consider first the case that the number of zero bits in $A_{i}$ remains the same (which happens with probability at most $e^{-\alpha c}+ O(\eps)$).
In this case we accept the changes within $R$ similarly as we did for those within $A_{i}$, as explained above. In other words: with probability at most $ e^{-\alpha c}(|R|/n)ce^{-(|R|/n) c}\cdot \eps + O(\delta)$ the number of zero bits within $R$ decreases by one, and the contribution of events where it may increase is $O(\eps^2)$. 

If the number of zero bits within $A_{i}$ decreases (which happens with probability $\alpha c e^{-\alpha c} \eps +O(\eps^2)$), then we accept any change within $R$. In this case the expected change in the number of zero bits within $R$ is $c|R|/n+O(\eps)$. 
Summarizing, we observe that the expected drift for the density within $R$ is  
\begin{align*}
 \frac{\alpha c e^{-\alpha c}\eps \cdot \frac{c|R|}{n}-e^{-\alpha c}\frac{|R|}{n}ce^{-(|R|/n) c}\cdot \eps+  O(\delta +\eps^2)}{|R|} 
 =\frac{\eps c}ne^{-\alpha c}  \left(\alpha c-  e^{-(1-\alpha) c}   \right)+ O(\tfrac1n (\delta + \eps^2+\beta)) =:\Delta_2,
 \end{align*}
where we used that $|R|/n = 1-\alpha +O(\beta)$.

Observe that, up to the error terms, the ratio $\Delta_2/|\Delta_1|$ is exactly the left hand side of  \eqref{eq:cond}. By making 
$\delta, \eps$ and $\beta$ small enough we may thus assume that 
\begin{equation}\label{eq:cond2}
\frac{\Delta_2}{|\Delta_1|} > \frac{\alpha}{1-\alpha} + O(\beta+\gamma).
 \end{equation} 
 

We now employ the tail bounds in Theorem~\ref{thm:tailbounds}: as we have $d(A_{i},t_{i}') \geq \eps+(1-\alpha)\delta-\beta-\gamma$, Theorem~\ref{thm:tailbounds} tells us that with probability $1-e^{-\Omega(n)}$ it takes at least $S=[(1-\alpha)\delta-\beta-3\gamma]/|\Delta_1|$ steps to decrease it to $\eps+\gamma$. Before this density is reached the density in $B_{i}$ is strictly larger than $\eps$ by~\eqref{eq:lateBi}. On the other hand, $d(R,t_{i}') \geq \eps+(1-\alpha)\delta - \beta-4\gamma$. Again by Theorem~\ref{thm:tailbounds} the density in $R$ is increased to at least $\eps+\delta+\gamma$ with probability $1-e^{-\Omega(n)}$ after at most $(\alpha\delta+\beta+6\gamma)/\Delta_2${\tiny $\stackrel{\eqref{eq:cond2}}{<}$}$S$ steps, and stays above $\eps+\delta$ by Theorem~\ref{thm:tailbounds}b. This proves both (ii) and (iii) of the induction and thus concludes the proof of the theorem. 
%
%
%
\end{proof}

A natural question is to determine the behaviour of the algorithm if the mutation parameter $c$ lies in the interval $[1,2.2)$. We leave this as an intriguing open problem.

\subsection{Linear Functions}\label{sec:linear}

In this section we give a new proof that the optimization time of algorithm EA with mutation parameter $c$ is $(1+o(1))\tfrac{e^{c}}{c}  n\log n$ for any constant $c>0$ if we restrict ourselves to \emph{linear} functions $f(\xi) = \sum_{i=1}^n a_i \xi_i$, as was first proven in~\cite{Witt2013}. 

\begin{theorem}\cite{Witt2013}\label{thm:linearfunctions}
For any linear function $f$ and every constant $c>0$, algorithm EA with mutation rate $c/n$ has optimization time $(1+o(1))e^{c}/c \cdot n\log n$ whp.
\end{theorem}

The lower bound in Theorem~\ref{thm:linearfunctions} is rather straightforward, and we only give a sketch. Assume the number $X_t$ of zero bits in $\xi^{t}$ is small, say $X_t \leq n^{1-\eps}$. Then the probability to flip at least two zero bits in the same round is $O(n^{-2\eps})$. This will turn out to be negligibly small, so let us ignore such ``bad'' rounds for the moment for the sake of clarity. Then the number of zero bits can only decrease if we flip exactly one zero bit, and no one bits. Let $r$ be the number of rounds in which we flip exactly one position. Then the number of such rounds until no zero bits are left is (if there were no bad rounds) the same as the number of rounds needed in a coupon collector process to reduce the number of missing coupons from $n^{1-\eps}$ to $0$, which is well-known to be at least $(1-\eps-o(1))n\log n$ whp. On the other hand, in each round we have probability $(1+o(1))ce^{-c}$ to flip exactly one bit. So the (total) number of rounds until we have seen whp $r= (1-\eps-o(1))n\log n$ rounds in which we flip only a single bit  is $(1-\eps-o(1))e^{c}/c \cdot n\log n$, by  Chernoff bounds. It is easy to adapt the argument to also cope with bad rounds, e.g., by dividing the process into phases form $n^{1-\eps}$ to $n^{1-2\eps}$, from $n^{1-2\eps}$ to $n^{1-4\eps}$ etc., and observing that in each phase with very high probability the number of bad rounds is so small that it does not significantly change the time to reach the next phase. 

\medskip

For the upper bound, we split the proof in two steps: we need a linear time until there are less than $\eps n$ zero bits left (Lemma~\ref{lem:startphase}), which are then removed in an additional $(1+o(1))e^c/c \cdot n\log n$ steps (Lemma~\ref{lem:endphase}). Surprisingly, the proof of the latter part is rather straightforward and uses known techniques, while the proof of Lemma~\ref{lem:startphase} is significantly harder. For this phase we present a new proof that is different from previous approaches: an induction-type argument on the amount of noise that is tolerable in noisy applications of algorithm EA (Lemma~\ref{lem:startphase2}).

\begin{lemma}\label{lem:endphase}
Let $c>0$ be a constant. Then there is a constant $\eps>0$ such that the following holds for every $\xi^0 \in \{0,1\}^n$ that has at most $\eps n$ zero bits. If algorithm EA with mutation rate $c/n$ starts with initial string $\xi^{0}$ then whp it finds the optimum of any $n$-bit linear function with positive weights after at most $(1+o(1))e^c/c \cdot n\log n$.
\end{lemma}

\begin{proof}
Without loss of generality we assume that the weights $a_i$ of the linear function are sorted and positive, i.e.,\ $a_1\geq \ldots \geq a_n >0$. To make the upcoming argument smoother, we will assume that in the special case that exactly one one-bit and one zero-bit are flipped, and both have equal weights, then the algorithm rejects the offspring. Note that we may make this assumption because the offspring and the parent are identical up to reordering the indices.

Our aim is to apply Theorem~\ref{thm:multiplicative}. For the function $g$ we 
count the number of zeroes, but we also distinguish between the positions where the zeroes occur, giving those with a higher coefficient $a_i$ also a higher weight. It seems natural to use the $a_i$ themselves as weights. However, this does not work since Theorem~\ref{thm:drift:t4} only gives a very bad bound if, for example, $a_1/a_n \gg n$. Instead, we fix a constant $C >0$, to be chosen later, and we use the objective function
\[
g(\xi) := 2^C\cdot \sum_{1 \leq i \leq n, \xi_i =0} 2^{-Ci/n}.
\]
Let $X_t$ be the number of zero bits in $\xi^{t}$, and observe that $X_t \leq g(\xi^{t}) \leq 2^CX_t$.

We want to compute the drift $\Exp[g(\xi^{t+1}) \mid \xi^{t}=\xi]- g(\xi)$, so assume $\xi^{t} = \xi$. First observe that the probability to flip at least two zero bits in the same round is $O((X_t/n)^2) = O(g(\xi)^2/n^2)$. Moreover, conditioning on flipping at least two zero bits, the expected number of bit flips is still $O(1)$, so this case contributes at most $O(g(\xi)^2/n^2)$ to the drift (which, as we will see, is negligible). On the other hand, if no zero bit is flipped then the offspring is rejected. So it remains to consider the case where exactly one zero bit is flipped (and possibly additional one bits are flipped). We call the index of this bit $k$. If no other bit is flipped, then we accept the offspring, and the probability of this case is $(1+o(1))ce^{-c}/n$ for any fixed $k$. Hence, this case contributes
\[
-(1+o(1))\frac{ce^{-c}}{n} 2^C \cdot\sum_{1 \leq k \leq n, \xi_{k}=0} 2^{-Ck/n} = -(1+o(1))\frac{ce^{-c}}{n}g(\xi)
\]
to the drift. It remains to consider the case where also one bits are flipped. If all one bits that are flipped have an index $i$ that satisfies $i >k$ then the offspring may or may not be accepted, but in any case the function $g$ increases by at most $S_k:=2^C\sum_{j > k, \xi_j \text{ flipped}} 2^{-Cj/n}$. On the other hand, if a one bit $i$ is flipped with $i<k$ then the fitness does not increase, so the offspring is rejected. (Here we use the assumption from the beginning of the proof). Hence, the change in $g$ is $0$. Note that $S_k$ is always non-negative, so we can upper bound the positive contribution to the drift from all cases where we flip exactly one zero bit by
\begin{align*}
\sum_{1 \leq k \leq n, \xi_k= 0} \Pr[\xi_k \text{ flips}] \cdot \Exp[S_k] & = \sum_{1 \leq k \leq n, \xi_k= 0} \frac{c}{n} \cdot 2^C\sum_{i=k+1}^n\frac{c}{n}2^{-Ci/n} 
\leq \frac{c2^C}{n}\sum_{1 \leq k \leq n, \xi_k= 0}2^{-Ck/n} \underbrace{\sum_{i=1}^\infty \frac{c}{n}2^{-Ci/n}}_{= c/(C\log 2) +o(1)} \\
& = (1+o(1))\frac{c^2}{nC\log 2}g(\xi).
 \end{align*}
Assume that $C$ is so large that $\delta := ce^c/(C\log 2) \leq 1/3$. Then, bringing all cases together, we obtain for any string $\xi$,
\begin{equation}\label{eq:linearendphase}
\Exp[g(\xi^{t+1}) \mid \xi^{t}=\xi] \leq g(\xi)\left(1- (1+o(1))\frac{(1-\delta)ce^{-c}}{n} +O\left(\frac{g(\xi)}{n^2}\right)\right),
\end{equation}
where the term $O(g(\xi)/n^2)$ covers the case that at least two zero bits are flipped.
This inequality implies two things. Firstly, we apply~\eqref{eq:linearendphase} for $C= 3ce^{c}/\log 2$ (i.e., $\delta = 1/3$). Since $g(\xi^{0}) \leq 2^C X_0 \leq 2^C\eps n$, we may choose $\eps$ such that the error term $O(g(\xi^0)/n^2)$ is at most $\frac13 ce^{-c}/n$ at time $t=0$. Then by Theorem~\ref{thm:tailbounds}, for every $\gamma>0$ whp $X_t \leq g(\xi^{t})$ becomes smaller than $\gamma$ in time $t_0 = O(n)$, and it stays so for an exponential number of steps. Having proven this, we may now apply~\eqref{eq:linearendphase} for any constant $C > 3ce^{c}/\ln 2$ (and thus, any constant $0<\delta < 1/3$). Since then $g(\xi^{t}) \leq 2^C X_t \leq  2^C\gamma n$, we may choose $\gamma$ small enough so that the error term $O(g(\xi^{t})/n^2)$ is at most $\delta ce^{-c}/n$ for $t \geq t_0$. Thus we obtain $\Exp[g(\xi^{t+1}) \mid \xi^{t}=\xi] \leq g(\xi)(1- (1-3\delta)ce^{-c}/n)$. Then by Theorem~\ref{thm:multiplicative}, whp we reach zero after at most $(1+o(1))\frac{ne^c}{(1-3\delta)c}\log (g(\xi^{t_0}))$ additional steps. Since $g(\xi^{t_0}) = O(n)$, and since we can choose $\delta$ arbitrarily small, this proves that whp the algorithm finds the optimum after $(1+o(1))e^c/c \cdot n\log n$ steps.

\end{proof}

It remains to show that the number of zeroes is at most $\eps n$ after a short (i.e., linear) time.

\begin{lemma}\label{lem:startphase}
For any $c, \eps >0$ there is $C>0$ such that after $Cn$ steps of algorithm EA on any $n$-bit linear function with positive weights, whp there are at most $\eps n$ zero bits left.
\end{lemma}

Even though Lemma~\ref{lem:startphase} implies that in absolute terms the first phase of algorithm EA takes only linear time, while the second phase requires an additional $\log$-factor, the proof of Lemma~\ref{lem:startphase} turns out to be particularly tricky. We will use an induction-type argument, but not over natural numbers, but rather over values of $c$. For technical reasons, we fix $c_{\max}>0$ sufficiently large (but otherwise arbitrary) and prove the statement for all $0< c < c_{\max}$. 
Note that since $c_{\max}$ is arbitrary, this nevertheless implies the desired result for all constants $c>0$. 

For the inductive argument to go through, we need to use a slightly stronger statement, which we spell out in Lemma~\ref{lem:startphase2}. We will allow two types of noise. We say that there is $\delta_1$-\emph{noise of type 1}, if in each round we flip a coin, and with probability $\delta_1$ an adversary may decide whether we go to the offspring or not, regardless of the objective value. We say that there is $\delta_2$-\emph{noise of type 2} if in each round we flip a coin, and with probability $1-\delta_2$ an adversary may decide on an arbitrary non-negative penalty by which the fitness of the offspring is reduced. Note that the adversary has to set the penalty without looking at the outcome of the offspring.

Note that Lemma~\ref{lem:startphase2} contains Lemma~\ref{lem:startphase} as a special case (the case without noise), so proving Lemma~\ref{lem:startphase2} also concludes the proof of Theorem~\ref{thm:linearfunctions}.


\begin{lemma}\label{lem:startphase2}
If $c_{\max}>0$ is a sufficiently large constant, then for all $0 <\eps \leq 1$ and all $0 < c < c_{\max}$ there is $C>0$ such that for all $n \geq 1$ the following holds. Consider  algorithm EA with arbitrary starting string, with mutation rate $c/n$, and with $\delta_1$-noise of type 1 and $\delta_2$-noise of type 2, where $\delta_1 = \delta_1(\eps,c):=\eps \exp\{-ce^{4c_{\max} + e^{5c_{\max}}}\}$ and $\delta_2 = \delta_2(c):= e^{2(c-c_{\max})}$. Then for any $n$-bit linear function with positive weights, whp for all $t \in [Cn, e^{n/C}]$ the $t$-th search point $\xi^t$ has at most $\eps n$ zero bits.
\end{lemma}

\begin{proof}
We first investigate the case that $c$ is small. Observe that for a fixed $0<\eps \leq 1$, and $c$ sufficiently small, the probability to flip more than one bit in the same round can be bounded by $O(c^2)$. Thus, if the number of zero bits is at least $\eps n/2$ then the drift in the number of zero bits is at most $-\delta_2 ce^{-c}\eps/4 + \delta_1 c+ O(c^2)$. Here  $ce^{-c}\eps/2$ is a lower bound (for large $n$) for the probability that we flip exactly one bit and that this bit is a zero bit and the factor $\delta_2-\delta_1\ge \delta_2/2$ is a lower bound for the probability that the noise of type 1 or 2 will not obstruct us from accepting this bit flip. The remaining two terms are upper bounds for the cases in which the number of zero bits increases.
That is, if $c_{\max}$ is sufficiently large then for $c = \eps e^{-4c_{\max}}$ this drift is negative, and the lemma follows for the pair $\eps, c$ from Theorem~\ref{thm:tailbounds}. In other words, for each fixed $0<\eps \leq 1$ the lemma holds for some $c>0$.

Assume now for the sake of contradiction that the lemma is false for some $0<\eps \leq 1$, and let $\eps_0>0$ be the supremum over all $\eps$ for which it is false.    
Then there exists an $\eps > \eps_0/2$ for which the lemma is false. From now on we will stick with this $\eps$ and derive a contradiction.  Let $c_0 = c_0(\eps)$ be the infimum over all $c$ for which the lemma fails for the pair $(\eps,c)$. Then $c_0 \geq \eps e^{-4c_{\max}}$ by the considerations above. We will show that there exists $\rho_c = \rho_c(\eps)>0$ such that the lemma holds for all pairs $(\eps,c)$ with $c_0 \leq c < \min\{c_0(1+\rho_c),c_{\max}\}$, thus yielding the desired contradiction. In the following we assume that $c$ is an arbitrary but fixed value in the interval 
$[c_0,\min\{c_0(1+\rho_c),c_{\max}\})$, where $\rho_c$ is defined after equation~\eqref{eq:eps:ass2} below.

Before going into details we explain the main idea of our proof. 
First note that the lemma is \emph{true} for $\eps^+:=2\eps$: if $2\eps >1$ this is trivial, otherwise this follows from the definition of $\eps_0$ and the fact that $\eps > \eps_0/2$.
Thus we know that algorithm EA (with parameter $p=c/n$) reduces the number of zeroes in linear time to $2\eps n$ (and will stay below this number for an exponential time). 
We will show that there exists $\alpha>0$ and $0 < \hat \eps< \eps$ so that after an additional linear number of steps, there are at most $\hat \eps \alpha n$ zero bits left among the first $\alpha n$ bits. The remaining part has only $\tilde n := (1-\alpha)n$ bits, thus the mutation rate in this part is $(1-\alpha)c/\tilde n$. We will choose $\alpha$ in such a way that $\tilde c := (1-\alpha)c  < c_0$, so we know from the choice of $c_0$ that the lemma is true for the pair $(\eps, \tilde c)$. Thus, the remaining string contains at most $\eps \tilde n$ zeros after linear time, at least if we would apply the algorithm only to that part of the string. In order to really derive a contradiction we must take into account that the initial $\alpha n$ part of the string provides noise for the optimization of the remaining part: if a zero bit in the first part is flipped into a one bit (and possibly other bits in this part are also flipped), then this gives noise of type~1; if one or several one bits are flipped then this may give noise of type~2. We thus need to check that both types of noise are sufficiently small. More precisely, we will check that this additional noise is small enough, so that the total noise is less than the values for $\delta_1$ and $\delta_2$ with respect to $\eps$ and $\tilde c$.

To prove these claims we need to define some constants.  Let $\delta_1 := \delta_1(\eps,c)$ and $\delta_2 := \delta_2(c)$ be as defined in the statement of the lemma. Moreover, let
\begin{equation}\label{eq:defhateps}
\hat \eps :=  e^{4c_{\max} + e^{5c_{\max}}}\delta_1 \qquad \text{ and } \qquad 
\beta := \delta_1.
\end{equation}
 Recall that $\eps e^{-4c_{\max}} \le c_0\le c < c_{\max}$. So we may assume without loss of generality that $c_{\max}$ is large enough so that 
\begin{equation}\label{eq:eps:ass1}
(2e\eps c\exp\{-e^{5c_{\max}}\})^{\exp\{e^{5c_{\max}}\}} \leq \delta_1 \quad \text{ and } \quad  c_{\max}\beta = c_{\max}\delta_1 \le {\textstyle\frac1{200}} \hat \eps \delta_2e^{-c_{\max}-e^{5c_{\max}}} \le {\textstyle\frac{1}{16}}.
\end{equation}
Note that \eqref{eq:eps:ass1} implies in particular that 
\begin{equation}\label{eq:eps:ass2}
1-c\beta \ge e^{-2c\beta}\qquad\text{and}\qquad 1-2c\beta \ge e^{-4c\beta}.
\end{equation}
With these definitions at hand, we let $\rho_c$ be so small that $c(1-\beta) < c_0$ whenever $c_0 \leq c \leq c_0(1+\rho_c)$. Note that $\beta$ implicitly depends on $c$ and $\eps$ here, but it is easy to check that the definition $\rho_c := \min_{\eps e^{-4c_{\max}} \leq c \leq c_{\max}}\{\beta(c)\}$ yields the desired inequalities. 

%
%
%
%
%
As before, we may assume that the bits are sorted by weights, in descending order. We call the first $\beta n$ bits \emph{the first region} $R_1$, the next $\beta n$ bits \emph{the second region} $R_2$, and the remaining $(1-2\beta)n$ bits the \emph{third region} $R_3$. We assume further without loss of generality that the maximum weight in the second section is $1$ (by scaling all $n$ weights by the same factor), and we split the second and third region into blocks as follows. For $i\geq 1$, the block $B_i$ consists of all bits in $R_2 \cup R_3$ with weight in $[(1+\mu)^{-i-1},(1+\mu)^{-i})$, where we will define $\mu$ in a moment. Note that some of the blocks may be empty. 
Recall that $d(B_i) = d(B_i,t)$ is the density of zero bits in $B_i$ with respect to $\xi^{t}$, for any non-empty block $B_i$, and similarly for $d(R_1)$ etc.

We distinguish two cases (in which we will choose $\alpha=\beta$ and $\alpha=2\beta$, respectively) depending on how many blocks are fully contained in the second region. Note that we also count empty blocks here, so the number of fully contained blocks in the second region is (up to rounding issues) $ \log_{1+\mu} 1/a_i$, where $i = 2\beta n$ is the first index in the third region.
To describe the cases precisely we need two more constants. Let 
\begin{equation}\label{eq:defmu}
D := \min\{D' \in \N_0 \mid (1+\mu)^{D'} \geq \exp\{e^{5c_{\max}}\}\}, \quad\text{where }\mu := \exp\{-c_{\max}-2e^{5c_{\max}}\}\delta_2\beta \hat\eps/100.
\end{equation}

\noindent
\textbf{First case}: there at least $D$ blocks fully contained in $R_2$. In this case we will set $\alpha = \beta$. Recall from above that we may assume that the number of zero bits is at most $2\eps n$. 
We consider the effect of the algorithm on $R_1$, i.e.\ on the first $\beta n$ bits. We want to argue that we can bring $d(R_1)$ below $\hat \eps$ in linear time, and that it stays so for exponential time. To this end, we will show that we have a negative drift whenever $d(R_1) \geq \hat \eps/2$, so let us assume $d(R_1) \geq \hat \eps/2$. The probability that we flip exactly one zero bit in $R_1$, and no other bit anywhere, is at least $ce^{-c_{\max}}\beta\hat \eps/4$ if $n$ is sufficiently large; in this case we accept this bit flip with probability at least $\delta_2-\delta_1 > \delta_2/2$, and the number of zero bits decreases by one. 
On the other hand, we can increase the number of zero bits in $R_1$ only if one of the following (non-disjoint) events occurs:
\begin{list}{}{\itemsep 0pt\topsep0pt\leftmargin0.4cm\labelsep 0.1cm\labelwidth0.3cm}
\item[\emph{(i)}\hfill]at least two bits in $R_1$ are flipped: this happens with probability at most $c^2\beta^2$ and, conditioned on this,  the expected increase is then at most $\beta c+2 \leq 3$, with room to spare;
\item[\emph{(ii)}\hfill]at least one bit is flipped in $R_2$: this happens with probability at most $c\beta$  and the expected increase (in $R_1$) is then at most $\beta c$; 
\item[\emph{(iii)}\hfill]more than $(1+\mu)^D$ zero bits are flipped in $R_3$: since there are at most $2\eps n$ zero bits, this probability is bounded by the probability that a binomially distributed random variable with parameters $2\eps n$ and $c/n$ is at least $(1+\mu)^D \geq \exp\{e^{5c_{\max}}\}$; using the Chernoff bound $\Pr[\Bin(N,q) \geq (1+\delta)Nq] \leq (e/(1+\delta))^{(1+\delta)Nq}$ with $Nq = 2\eps c$ and $(1+\delta)Nq = \exp\{e^{5c_{\max}}\} $ we see that this probability is at most $(2e\eps c\cdot \exp\{-e^{5c_{\max}}\})^{\exp\{e^{5c_{\max}}\}}$;
the expected increase in this case is again at most $\beta c$; 
\item[\emph{(iv)}\hfill]there is noise of type 1: this happens with probability $\delta_1$ and also results in an expected increase of at most $\beta c$. 
\end{list}
Summarizing the above terms, we see that the expected drift for the number of zero bits in the first region can be bounded by
\[
\Exp[X_{t+1}-X_t] \le - {\textstyle\frac18}\delta_2ce^{-c_{\max}}\beta \hat\eps + 3c^2\beta^2+c^2\beta^2 +(2e\eps c\exp\{-e^{5c_{\max}}\})^{\exp\{e^{5c_{\max}}\}} \beta c + c\beta\delta_1 \stackrel{\eqref{eq:eps:ass1}}= -\Omega(1).
\]
Hence, by Theorem~\ref{thm:tailbounds} after linear time we will reach $d(R_1) \leq \hat \eps/2$, and since we have negative drift whenever $d(R_1)$ is in the interval $[\hat \eps/2, \hat \eps]$, it will stay below $\hat \eps$ for exponential time.

We want to apply the lemma for $\eps$ 
and $\tilde c = c(1-\beta) < c_0$ to conclude that in this period $d(R_2 \cup R_3)$ drops below $\eps$. So we must show that the additional noise caused by the first region is not too large. Noise of type 1 is generated whenever a zero bit in $R_1$ is flipped, which happens with probability at most $c\beta\hat\eps$. Thus the total amount of noise of type 1 is at most
\[
\delta_1(\eps,c) + c\beta \hat \eps \stackrel{\eqref{eq:defhateps}}{\leq} \delta_1(\eps,c) + c\beta \frac{-\partial\delta_1(\eps, c)}{\partial c} \leq \delta_1(\eps,\tilde c),
\]
where in the last step we used convexity of $\delta_1$. 
For noise of the type 2, the fitness of the offspring may decrease if any bit in $R_1$ is flipped, which happens with probability at most $c \beta$. Hence, the overall probability that the adversary will {\em not} add a penalty is at least
\[
\delta_2(c)\cdot(1-c\beta) \stackrel{\eqref{eq:eps:ass2}}{\geq} \delta_2(\tilde c)e^{2c\beta}\cdot e^{-2c\beta} = \delta_2(\tilde c),
\]
as required. This concludes the first case.\smallskip

\textbf{Second case}: there are less than $D$ blocks fully contained in the second region. In this case we will apply the lemma to the third region ($\alpha = 2\beta$) with parameters $\eps$ and $\tilde c  = c(1-2\beta) < c_0$. To bound the noise coming from $R_1$ and $R_2$, we will first show that whp $d(R_2) \leq \hat \eps$ after linear time. 
We use the auxiliary fitness function 
\[
\phi(\xi) := \sum_{i\geq 1} |\{j \in B_i: \xi_j =0\}| \cdot (1+\mu)^{-i},
\]
%
and we let $\phi_{\max}$ be the fitness of the zero string, i.e. $\phi_{\max} = \sum_{i\geq 1} |B_i| \cdot (1+\mu)^{-i}$. Note that by definition of the second case $\phi_{\max} \geq \beta n \cdot (1+\mu)^{-D} \geq  \beta n \cdot \exp\{-e^{5c_{\max}}\}/(1+\mu)$ by definition of $D$. 

We call a block in the third region \emph{long} if it has length at least $\beta \mu n/2$, and \emph{short} otherwise. The contribution of short blocks to $\phi_{\max}$ is at most
\begin{equation}\label{eq:shortblocks}
\sum_{B_i \text{ short}} |B_i|(1+\mu)^{-i} < \frac{1}{2} \beta \mu n \sum_{i \geq 0} (1+\mu)^{-i} < \beta n.
\end{equation}
As long blocks have linear length, there can only be a constant number of long blocks. Thus we may apply Lemma~\ref{lem:lindensity} to conclude that (after a linear number of steps) we have $d(B_i,t) \geq d(R_2,t)/2$ for all long blocks $B_i$ and for exponential time. Note that the contribution of a block $B_i$ to $\phi(\xi^{t})$ can be written as $d(B_i,t)|B_i|(1+\mu)^{-i}$, and the contribution of $R_2$ is at least $d(R_2,t)\beta n(1+\mu)^{-D}$, where the later follows from the definition of the second case. Hence,
\begin{align}
\phi(\xi^{t})  \geq \frac{d(R_2,t) \beta n}{(1+\mu)^{D}} + \sum_{B_i \text{ long}}\text{ $\frac12$} d(R_2,t)|B_i|(1+\mu)^{-i} & > \frac{d(R_2,t)}{2(1+\mu)^{D}}\left(2 \beta n + \sum_{B_i \text{ long}} |B_i|(1+\mu)^{-i} \right) \nonumber \\ & \stackrel{\eqref{eq:shortblocks}}{\geq} \frac{d(R_2,t)\phi_{\max}}{2(1+\mu)^{D}}.\label{eq:xit}
\end{align}

Now we are ready to show that $\phi$ has a negative drift for $d(R_2,t)\ge \hat\eps/2$, so assume the latter. First observe that with probability $ce^{-c}(1+o(1))$ there is exactly one bit flipped, and there is no noise with probability at least $\delta_2-\delta_1 > \delta_2/2$. In this case each bit has probability $1/n$ to be the flipped one, so in this case we decrease $\phi(\xi^{t})$ by $\phi(\xi^{t})/n$ on average. Thus the contribution to the drift is (for sufficiently large $n$) more negative than $-ce^{-c}\delta_2\phi(\xi^{t})/(4n) \leq -ce^{-c}\delta_2\hat\eps (1+\mu)^{-D}\phi_{\max}/(16n)$ by~\eqref{eq:xit}. If at least one bit in $R_1$ is flipped (which happens with probability at most $c\beta$), then this increases $\phi(\xi^{t})$ in expectation by at most $c\phi_{\max}/n$, so this adds at most $c^2\beta\phi_{\max}/n$. Similarly, noise of type 1 contributes at most $c\delta_1 \phi_{\max}/n$ to the drift. 
Finally, consider the case where no bit in $R_1$ flips, and where no noise of type 1 occurs. Let $r_i$ and $s_i$ be the number of bits in the $i$-th block $B_i$ that flip from zero to one and from one to zero, respectively. Then a necessary (but not sufficient) condition for accepting the offspring is $\sum_{i} s_i(1+\mu)^{-i-1} - \sum r_i(1+\mu)^{-i} \leq 0$. Note that this condition is still necessary if noise of type 2 occurs. Thus, in this case $\Delta \phi = \sum_{i} s_i(1+\mu)^{-i} - \sum r_i(1+\mu)^{-i} \leq \mu \sum_{i}r_i(1+\mu)^{-i}\leq \mu \sum_{i}r_i$, which is at most $c\mu$ in expectation. Summarizing, we see that for $d(R_2,t)\ge \hat\eps/2$ the drift is at most
\[
\Exp[\phi(\xi^{t+1})-\phi(\xi^{t})] \leq -ce^{-c}\delta_2\hat\eps (1+\mu)^{-D}\phi_{\max}/(16n) + c^2\beta \phi_{\max}/n  +c\delta_1 \phi_{\max}/n + c\mu \stackrel{\eqref{eq:eps:ass1}, \eqref{eq:defmu}}{=} -\Omega(1)
\]
for sufficiently large $c_{\max}$. 

As $d(R_2,t) > \hat\eps/2$ implies  a negative drift for $\phi(\xi^{t})$, Theorem~\ref{thm:tailbounds} implies that after a linear number of steps we have $d(R_2) \leq \tfrac{3}{4}\hat\eps$, and that this bounds remains true for an exponential number of steps. By Lemma~\ref{lem:lindensity} the density in $R_1$ is, at least after another linear number of steps, at most $\frac{4}5\hat\eps$. Combining both facts we conclude that from then on we have $d(R_1 \cup R_2) \leq \hat \eps$ for an exponential number of steps. 

We now apply the lemma similarly as in the first case, this time to $R_3$. I.e., we set $\alpha = 2\beta$, and apply the lemma for $\eps$ and $\tilde c = c(1-2\beta) < c_0$. We need to check that we do not introduce too much noise. For noise of type 1, we have at most
\[
\delta_1(\eps,c) + 2c\beta\hat\eps \stackrel{\eqref{eq:defhateps}}{\leq} \delta_1(\eps,c) + 2c\beta \frac{-\partial\delta_1(\eps, c)}{\partial c} \leq \delta_1(\eps,\tilde c),
\]
where again we used convexity in the last step. For noise of the type 2 we conclude as in the first case that the probability \emph{not} to have noise is at least
\[
\delta_2(c)(1-2c\beta) \stackrel{\eqref{eq:eps:ass2}}{\geq}  \delta_2(\tilde c)e^{4c\beta}\cdot e^{-4c\beta} \geq \delta_2(\tilde c),
\]
as desired.
\end{proof}

\paragraph{Acknowledgement.} After being appointed at ETH Z\"urich, Ji\v{r}\'i Matou\v{s}ek convinced his colleagues in the algorithms group to establish a joint seminar to teach undergraduate students  to do (and love) research. In the third edition of this seminar the second author of this paper presented the topic at hand. More precisely, the task was to come up with fresh ideas to elegantly reprove the state of the art. While this turned out to be a too hard task for undergraduate students, it eventually led to the present paper, which for us will always be a memory for our much admired colleague Jirka Matou\v{s}ek.


\begin{thebibliography}{10}

\bibitem{back1996evolutionary}
Thomas B{\"a}ck.
\newblock {\em Evolutionary algorithms in theory and practice: evolution
  strategies, evolutionary programming, genetic algorithms}.
\newblock {Oxford University Press}, 1996.

\bibitem{DoerrGoldberg:j:11:adaptiveDrift}
Benjamin Doerr and Leslie~Ann Goldberg.
\newblock Adaptive drift analysis.
\newblock {\em Algorithmica}, 65(1):224--250, 2013.

\bibitem{doerr2013mutation}
Benjamin Doerr, Thomas Jansen, Dirk Sudholt, Carola Winzen, and Christine
  Zarges.
\newblock Mutation rate matters even when optimizing monotonic functions.
\newblock {\em Evolutionary Computation}, 21(1):1--27, 2013.

\bibitem{Doe-Joh-Win:j:12:multiDrift}
Benjamin Doerr, Daniel Johannsen, and Carola Winzen.
\newblock Multiplicative drift analysis.
\newblock {\em Algorithmica}, 64(4):673--697, 2012.

\bibitem{doerr2012non}
Benjamin Doerr, Daniel Johannsen, and Carola Winzen.
\newblock Non-existence of linear universal drift functions.
\newblock {\em Theoretical Computer Science}, 436(1):71--86, 2012.

\bibitem{Droste2002}
Stefan Droste, Thomas Jansen, and Ingo Wegener.
\newblock On the analysis of the (1+1) evolutionary algorithm.
\newblock {\em Theoretical Computer Science}, 287(1):131--144, 2002.

\bibitem{Haj:j:82}
Bruce Hajek.
\newblock Hitting-time and occupation-time bounds implied by drift analysis
  with applications.
\newblock {\em Advances in Applied Probability}, 13(3):502--525, 1982.

\bibitem{HeY01}
Jun He and Xin Yao.
\newblock Drift analysis and average time complexity of evolutionary
  algorithms.
\newblock {\em Artificial Intelligence}, 127(1):57--85, 2001.

\bibitem{HeYao:04:drift}
Jun He and Xin Yao.
\newblock A study of drift analysis for estimating computation time of
  evolutionary algorithms.
\newblock {\em Natural Computing}, 3(1):21--35, 2004.

\bibitem{jagerskupper2011combining}
Jens J{\"a}gersk{\"u}pper.
\newblock Combining {M}arkov-chain analysis and drift analysis.
\newblock {\em Algorithmica}, 59(3):409--424, 2011.

\bibitem{jansen2007brittleness}
Thomas Jansen.
\newblock On the brittleness of evolutionary algorithms.
\newblock In {\em Proceedings of the 9th International Workshop on Foundations
  of Genetic Algorithms (FOGA 2007)}, pages 54--69. Springer, 2007.

\bibitem{Joh:th:10}
Daniel Johannsen.
\newblock {\em Random Combinatorial Structures and Randomized Search
  Heuristics}.
\newblock PhD thesis, Universit{\"a}t des Saarlandes, 2010.

\bibitem{kotzing2016concentration}
Timo K{\"o}tzing.
\newblock Concentration of first hitting times under additive drift.
\newblock {\em Algorithmica}, 75(3):490--506, 2016.

\bibitem{Mit-Row-Can:j:09}
Boris Mitavskiy, Jonathan Rowe, and Chris Cannings.
\newblock Theoretical analysis of local search strategies to optimize network
  communication subject to preserving the total number of links.
\newblock {\em International Journal of Intelligent Computing and Cybernetics},
  2(2):243--284, 2009.

\bibitem{Muehlenbein92}
Heinz M{\"u}hlenbein.
\newblock How genetic algorithms really work: Mutation and hillclimbing.
\newblock In {\em Proceedings of the 2nd International Conference on Parallel
  Problem Solving from Nature (PPSN 1992)}, pages 15--26. Elsevier, 1992.

\bibitem{oliveto2012erratum}
Pietro~S Oliveto and Carsten Witt.
\newblock Erratum: Simplified drift analysis for proving lower bounds in
  evolutionary computation.
\newblock {\em arXiv preprint arXiv:1211.7184}, 2012.

\bibitem{Oli-Wit:j:11:negativeDrift}
Pietro~Simone Oliveto and Carsten Witt.
\newblock Simplified drift analysis for proving lower bounds in evolutionary
  computation.
\newblock {\em Algorithmica}, 59(3):369--386, 2011.

\bibitem{Witt2013}
Carsten Witt.
\newblock Tight bounds on the optimization time of a randomized search
  heuristic on linear functions.
\newblock {\em Combinatorics, Probability {\&} Computing}, 22(2):294--318,
  2013.

\end{thebibliography}

\end{document}